\documentclass{amsart}

\usepackage{amssymb,amsmath,latexsym}

\newtheorem{theorem}{Theorem}[section]
\newtheorem{lemma}[theorem]{Lemma}

\newtheorem{proposition}[theorem]{Proposition}
\newtheorem{example}[theorem]{Example}

\newcommand\ds{\displaystyle}
\newcommand{\Hil}{\mathcal H}
\newcommand{\R}{\mathbb R}

\newcommand{\C}{\mathbb C}

\newcommand{\N}{\mathbb N}

\newcommand{\Mn}{M_n(\C)}
\newcommand{\calA}{\mathcal A}

\newcommand{\lip}{\left<}
\newcommand{\rip}{\right>}
\newcommand{\tr}{\operatorname{trace}}
\newcommand{\Log}{\operatorname{Log}}

\def\conj#1{\overline{#1}}

\title[Inner Products]{Inner Products on the Space of \\ Complex Square Matrices}

\author{Rub\'en A. Mart\'inez-Avenda\~no}
\author{Josu\'e I. Rios-Cangas}

\address{Centro de Investigaci\'on en Matem\'aticas,  Universidad
  Aut\'onoma del Estado de Hidalgo, Ciudad del Conocimiento,
  Carr.~Pachuca--Tulancingo, km 4.5, Mineral de Reforma, Hidalgo, 42184, Mexico}

\email{rubeno71@gmail.com}
\email{bbk\_matikl@hotmail.com}
\keywords{inner products, positivity of transformations}

\subjclass[2010]{15A60,15A63,15A04}

\thanks{The authors would like to thank the referee for the careful
  revision of the paper. The first author would like to thank Jaime
  Cruz-Sampedro: his suggestion of finding examples of norms on the
  space of complex square matrices led to this investigation.}

\begin{document}

\begin{abstract}
In this paper we study the problem of finding explicit expressions for
inner products on the space of complex square matrices $\Mn$. We show that,
given an inner product $\lip \cdot, \cdot \rip$ on $\Mn$, with some
conditions, there exist positive matrices $A_j$ and $B_j \in \Mn$,
for $j=1, 2\dots, m$ such that
$$
\lip X, Y \rip = \sum_{j=1}^m \tr\left(Y^* A_j X B_j \right),
$$
for all $X, Y \in \Mn$. However, we show that the result does not hold for all inner
products. In fact, if the above expression does not hold, we show that
there exist positive matrices $A_j$ and $B_j \in \Mn$, for $j=1,
2\dots, m$ such that
$$
\lip X, Y \rip = -\tr\left(Y^* A_1 X B_1 \right)+ \sum_{j=2}^m
\tr\left(Y^* A_j X B_j \right),
$$
for all $X, Y \in \Mn$.
\end{abstract}

\maketitle 

\section{Introduction}

It is a well-known fact, and a very useful one, that on a finite
dimensional normed vector space all norms are equivalent. Of course,
that does not mean that all norms are equally easy to use. Among the
most easy-to-use norms are those induced by an inner product. 

The space $\Mn$ of all $n \times n$ matrices with complex entries is a vector space and
one can give several norms on it: some which are induced by an inner product,
some which are not. For example, the operator norm is not induced by an
inner product, which can be easily checked by observing that the
parallelogram law does not hold. On the other hand, the inner product
defined by
$$
\lip X, Y \rip = \tr(Y^* X),
$$
for $X, Y \in \Mn$ induces a norm called the Frobenius or Euclidean
norm given by
$$
\| X \|^2 = \tr(X^*X)=\sum_{i=1}^n \sum_{j=1}^n |x_{i,j}|^2
$$
where $X=\begin{pmatrix} x_{i,j} \end{pmatrix}$. Here $\tr$ denotes
the trace of the given matrix and $Y^*$ denotes the conjugate-transpose of $Y$.

A natural question arises: are there any other inner products on the
vector space of square complex matrices? It is not hard to show that
if we choose $m \in \N$ and positive matrices $A_i$, $B_i$ with $i=1, 2, \dots, m$ then
$$
\lip X, Y \rip_* = \tr \left( \sum_{i=1}^m Y^* A_i X B_i \right)
$$
defines an inner product. Are there any others?

The purpose of this paper is to show that, given an inner product
$\lip \cdot, \cdot \rip_*$ on $\Mn$, there exists $m \in \N$ and
positive matrices $A_i$, $B_i$ with $i=1, 2, \dots, m$ such that
$$
\lip X, Y \rip_* = \tr\left( - Y^* A_1 X B_1+ \sum_{i=2}^m Y^* A_i X B_i\right).
$$
The appearance of the minus sign in the first term is unsettling but,
nevertheless, unavoidable. We will give an example (a restatement of
an example shown to us by D.~Grinberg in \cite{grinberg}) that shows it is
not always possible to get rid of the minus sign. We also give a
theorem that guarantees that, under some conditions, given an inner product
$\lip \cdot, \cdot \rip_*$ on $\Mn$, there exists $m \in \N$ and
positive matrices $A_i$, $B_i$ with $i=1, 2, \dots, m$ such that
$$
\lip X, Y \rip_* = \tr\left(\sum_{i=1}^m Y^* A_i X B_i\right).
$$

We do not know if the above results are known or not. Maybe they can be
deduced from some heavy machinery. Nevertheless, we strive here
to give elementary proofs of our results.

The paper is organized as follows. In Section~\ref{sec:pre} we
introduce some results that we will use throughout the
paper. In Section~\ref{sec:her} we show that the matrices $A_i$ and
$B_i$ that define the inner product can always be chosen to be
Hermitian. In Section~\ref{sec:pos} we show that we can always choose
the matrices to be positive, except perhaps for the appearance of one
minus sign.

\section{Preliminaries}\label{sec:pre}

Recall that $\Mn$ is a Hilbert space under the inner product given by
$$
\lip X, Y \rip = \tr(Y^* X).
$$
We will assume throughout this paper that inner products are linear in
the first entry and conjugate-linear in the second entry. The same
assumption will be made about sesquilinear forms.

Given a bounded sesquilinear form $\phi: \Hil \times \Hil \to \C$ on a
Hilbert space $\Hil$, there exists (see~\cite[Theorem~II.2.2]{conway})
a unique bounded linear operator $\calA : \Hil \to \Hil$ such that
$$
\phi(x,y)=\lip \calA x, y \rip.
$$

Since any given inner product $\langle \cdot , \cdot  \rangle_*$ on
$\Mn$ is a bounded sesquilinear form on $\Mn$, it follows that said
inner product is necessarily of the form 
$$
\langle X, Y \rangle_* = \lip \calA(X), Y \rip = \tr(Y^*\calA(X)),
$$
with $\calA : \Mn \to \Mn$ a (bounded) linear transformation.

Also, it is well-known that, for any linear transformation $\calA :
\Mn \to \Mn$, there exist $m \in \N$ with $m \leq n^2$, and matrices $E_i,
F_i \in \Mn$, for $i=1, 2, \dots, m$, such that
$$
\calA(X) = \sum_{i=1}^m E_i X F_i.
$$
It is easy to see that the sets $\{E_1, E_2, \dots,E_m \}$ and $\{F_1,
F_2, \dots, F_m \}$ can be assumed linearly independent.

Therefore, the problem of characterizing the inner products on $\Mn$
reduces to the problem of characterizing all matrices $E_i, F_i \in
\Mn$ such that the linear transformation $\calA : \Mn \to \Mn$ defined
by
$$
\calA(X) = \sum_{i=1}^m E_i X F_i
$$
gives rise to a linear product of the form
$$
\langle X, Y \rangle_* = \lip \calA(X), Y \rip = \tr(Y^*\calA(X)).
$$

We start with results we will need throughout this paper. We think of
vectors $x \in \C^n$ as column vectors and thus $x^*$ is a row vector. Recall that
a matrix $A \in \Mn$ is said to be Hermitian if $A=A^*$. The matrix
$A$ is said to be positive if it satisfies $x^* A x >0$ for all nonzero $x \in
\C^n$. Equivalently, $A$ is positive if it is Hermitian and all of its
eigenvalues are positive. Also, recall that if $A$ and $B$ are
positive matrices, and if we denote by $\lambda(t)$ the smallest
eigenvalue of the matrix $A-tB$, then $\lambda(t)$ is a continuous
function of $t \in \R$. Also, for large enough $t$, the matrix $A-tB$ has a
negative eigenvalue. Hence, by the Intermediate Value Theorem, there
exists $t_0>0$ such that $A-tB$ is positive for all $t\in [0,t_0)$ and
$A-t_0B$ has a zero eigenvalue.

Also, it is well known that the set of positive matrices is an open
subset of the Hermitian matrices. Hence if $A$ is positive and $B$ is
Hermitian, there exists $\varepsilon >0$ such that $A+tB$ is positive
for all real numbers $t$ with $|t|<\varepsilon$.  We will use all the
facts above without further discussion.

An important property of the trace which we will use is that
$\tr(AB)=\tr(BA)$ and hence $\tr(ABC)=\tr(BCA)=\tr(CAB)$, as long as
the matrix multiplications make sense. Incidentally, $\tr(ABC)$ does
not have to be equal to $\tr(BAC)=\tr(CBA)=\tr(ACB)$.

The following easy lemma will be useful.

\begin{lemma}\label{le:tra_pos}
Let $A$ and $B$ be positive matrices in $\Mn$. Then
$$
\tr(X^* A X B) >0
$$
for all nonzero $X \in \Mn$.
\end{lemma}
\begin{proof}
Since $A$ and $B$ are positive, there exist unitary matrices $U$ and
$V$ such that $A':=U^*AU$ and $B'=V^*B V$ are diagonal matrices, say
$A'=\mathrm{diag}(\lambda_1, \lambda_2, \dots, \lambda_n)$ and 
$B'=\mathrm{diag}(\mu_1, \mu_2, \dots, \mu_n)$ with $\lambda_i >0 $
and $\mu_i>0$ for all $i=1, 2, \dots, n$.

Then for every nonzero $X \in \Mn$ we have
\begin{align*}
\tr(X^*AXB)
&=\tr(X^*UA'U^* X V B'V^*) \\
&=\tr(A' (U^* X V) B' (V^* X^* U))\\
&=\tr(A'Y B' Y^*),
\end{align*}
where $Y= U^* X V$ (note that $Y\neq 0$). But if we denote the entries
of $Y$ by $\begin{pmatrix} y_{i,j} \end{pmatrix}$, a simple computation shows that
$$
\tr(X^*AXB) = \tr(A'Y B' Y^*) = \sum_{i,j=1}^n \lambda_i \mu_j
|y_{i,j}|^2 > 0,
$$
and hence $\tr(X^*AXB) > 0$.
\end{proof}

The following lemma will allow us to show the positivity of the
matrices defining the operator $\calA$.

\begin{lemma}
Let $\calA : \Mn \to \Mn$ be defined by
$$
\calA(X)=\sum_{i=1}^m E_i X F_i.
$$
Then,
$$
\lip \calA (x y^*), x y^* \rip= \sum_{i=1}^m  (x^* E_i x)(y^* F_i y)
$$
for all $x, y \in \C^n$. 
\end{lemma}
\begin{proof}
Let $E, F \in \Mn$, and define $X:=x y^* \in \Mn$. Then,
\begin{align*}
\lip E X F, X \rip  
&= \tr( X^* E X F)\\
&= \tr( y x^* E x y^* F)\\
&= \tr( x^* E x y^* F y)\\
&= (x^* E x)(y^* F y).
\end{align*}
Applying this result to each term of the form $\lip E_i X F_i, X \rip$ the result
follows by linearity.
\end{proof}

The following result is just a well-known result about inner product spaces. We
record here the statement for later reference.

\begin{proposition}\label{prop:ortho}
If $\tr(Y^* A)=\tr(Y^*B)$ for all $Y \in \Mn$, then $A=B$.
\end{proposition}

In Section~\ref{sec:her} we will need the existence of square roots of
matrices that satisfy certain properties. The following theorem
takes care of that. Recall that if $A \in \Mn$ is given by
$A=\begin{pmatrix} a_{i,j} \end{pmatrix}$ we define $\conj{A}$ as
the matrix $\conj{A}=\begin{pmatrix} \conj{a_{i,j}} \end{pmatrix}$.

\begin{theorem}\label{th:coninv}
Let $C \in \Mn$ be such that $C \conj{C}=I$. Then there exists a
matrix $D \in \Mn$ such that $D^2=C$ and $D \conj{D}=I$.
\end{theorem}
\begin{proof}
Given that $C$ is invertible there exists a primary matrix function
$\Log(C)$ such that $\Log\left(\conj{C^{-1}}\right)=-\conj{\Log(C)}$
and with $e^{\Log(C)}=C$ (see~\cite[Theorem 6.4.20]{HJ}). Since
$\conj{C}=C^{-1}$ it follows that
$$
\Log(C)+\conj{\Log(C)}=0
$$
and hence that $\Log(C)$ has purely imaginary entries.

Define $D=e^{\frac{1}{2} \Log(C)}$. Since $\Log(C)$ is purely
imaginary, then $\conj{D}=e^{-\frac{1}{2} \Log(C)}$. Hence $D
\conj{D}=e^{\frac{1}{2} \Log(C)}e^{-\frac{1}{2} \Log(C)}=I$ and
$D^2=e^{\Log(C)}=C$ as desired.
\end{proof}

In~\cite{FS}, Fong and Sourour prove the following two theorems (in
their paper it is just one theorem) in a more general context. We
state here the version of the theorems we will need.

\begin{theorem}[Fong--Sourour\cite{FS}]\label{th:fs1}
Let $A_j, B_j \in \Mn$ for $j=1, 2, \dots, m$. Assume that the set
$\{B_1, B_2, \dots, B_m \}$ is linearly independent and define the linear
map  $\Phi: \Mn \to \Mn$ by
$$
\Phi(X):=\sum_{j=1}^m A_j X B_j.
$$
Then $\Phi$ is identically zero if and only if $A_j=0$ for all $j$.
\end{theorem}

In the theorem above, what happens if the set $\{B_1, B_2, \dots, B_m
\}$ is not linearly independent? By renaming the indices, we may
assume that there exists an index $s<m$ such that the set $\{B_1, B_2,
\dots, B_s \}$ is linearly independent and generates the linear span
of $\{B_1, B_2, \dots, B_m \}$. Of course, we also rename the matrices
in the set  $\{A_1, A_2, \dots, A_m \}$ in order to keep the same
transformation $\Phi$. In the statement of the following theorem, we
assume we have done this reordering

\begin{theorem}[Fong--Sourour\cite{FS}]\label{th:fs2}
Let $A_j, B_j \in \Mn$ for $j=1, 2, \dots, m$. Assume that the set $\{B_1,
B_2, \dots, B_s \}$ is linearly independent for some $s<m$ and it
generates the linear span of the set $\{ B_1, B_2, \dots, B_m \}$. Let
$$
B_j=\sum_{k=1}^s c_{k,j} B_k
$$
for constants $c_{k,j}\in \C$ and $s+1 \leq j \leq m$. Define the linear
map $\Phi : \Mn \to \Mn$ by
$$
\Phi(X):=\sum_{j=1}^m A_j X B_j.
$$
Then, $\Phi$ is identically zero if and only if
$$
A_k = - \sum_{j=s+1}^m c_{k,j} A_j
$$
for $1 \leq k \leq s$.
\end{theorem}

The above two theorems will be used throughout this paper. The
following lemma is a consequence of them.

\begin{lemma}\label{le:eq}
Let $A_j, B_j \in \Mn$ for $j=1, 2, \dots, m$ such that $\{ A_1, A_2,
\dots, A_m\}$ is a linearly independent set and $\{ B_1, B_2, \dots,
B_m\}$ is a linearly independent set. Assume that 
$$
\sum_{j=1}^m A_j X B_j =\sum_{j=1}^m A_j^* X B_j^* 
$$
for all $X \in \Mn$. Then, there exist constants $c_{k,j}$ such that
$$
A_k^*=\sum_{j=1}^m \conj{c_{k,j}} A_j
$$
and 
$$
B_k^*=\sum_{j=1}^m c_{j,k} B_j,
$$
for $k=1, 2, \dots, m$.
\end{lemma}
\begin{proof}
Define $\Phi:\Mn \to \Mn$ as
$$
\Phi(X):=\sum_{j=1}^m A_j X B_j - \sum_{j=1}^m A_j^* X B_j^*.
$$
By hypothesis, $\Phi$ is identically zero. If 
$$
\{ B_1, B_2, \dots, B_m, B_1^*, B_2^*, \dots, B_m^*\}
$$ 
is linearly independent, then, by Theorem~\ref{th:fs1}, it follows
that $A_j=0$ for all $j$, which contradicts the hypothesis. By
reordering (simultaneously) the indices of the sets $\{ B_1, B_2, \dots, B_m\}$
and $\{ A_1, A_2, \dots, A_m \}$ we may assume that the set
$$
\{ B_1, B_2, \dots, B_m, B_1^*, \dots, B_s^*\}
$$ 
is linearly independent for some $1 \leq s < m$, and it generates the
span of 
$$
\{ B_1, B_2, \dots, B_m, B_1^*, B_2^*, \dots, B_m^*\}.
$$
It then follows that
$$
B_j^*=\sum_{k=1}^m t_{k,j} B_k + \sum_{k=1}^s t_{k,j}^*B_k^*
$$
for some constants $t_{k,j}$ and $t_{k,j}^*$ (not all zero) for $s+1 \leq j \leq
m$. By Theorem~\ref{th:fs2}, we have that
$$
A_k=\sum_{j=s+1}^m t_{k,j} A_j^*
$$
for $1 \leq k \leq m$ and
$$
A_k^*= - \sum_{j=s+1}^m t_{k,j}^* A_j^*
$$
for $1 \leq k \leq s$. This last equation can be written as
$$
A_k= - \sum_{j=s+1}^m  \conj{t_{k,j}^*} A_j
$$
which contradicts the linear independence of the set $\{ A_1, A_2,
\dots, A_m \}$.

From this, we conclude that the set 
$$
\{ B_1, B_2, \dots, B_m\}
$$
is linearly independent and each element in $\{ B_1^*, B_2^*, \dots,
B_m^*\}$ is a linear combination of the elements in $\{ B_1, B_2,
\dots, B_m\}$. Thus
$$
B_j^*=\sum_{k=1}^m c_{k,j} B_k 
$$
for some constants $c_{k,j}$ for $1 \leq j \leq m$ and hence, by
Theorem~\ref{th:fs2},
$$
A_k=\sum_{j=1}^m c_{k,j} A_j^* 
$$
and hence
$$
A_k^*=\sum_{j=1}^m \conj{c_{k,j}} A_j,
$$
as desired.
\end{proof}


\section{Hermitian Matrices}\label{sec:her}

It is clear that if we have a collection of Hermitian matrices $A_i$,
$B_i$, for $i=1, 2, \dots, m$, then the operator $\calA$ defined by
$$
\calA (X) = \sum_{i=1}^m A_i X B_i 
$$
has the property that
$$
\lip \calA (X), Y \rip = \lip X, \calA (Y) \rip 
$$
for all nozero $X \in \Mn$. Indeed,
\begin{align*}
\lip \calA (X), Y \rip 
&= \tr \left( Y^* \left(\sum_{i=1}^m A_i X B_i \right) \right) \\
&= \tr \left( \sum_{i=1}^m Y^* A_i X B_i \right) \\
&= \tr \left( \sum_{i=1}^m B_i Y^* A_i X \right) \\
&= \tr \left(\left(\sum_{i=1}^m A_i Y B_i \right)^* X \right) \\
&=\lip X, \calA(Y) \rip.
\end{align*}

In this section we will show that, given the linear transformation
$\calA: \Mn \to \Mn$ defined by
$$
\calA(X)=\sum_{i=1}^m A_i X B_i
$$
with $\lip X, Y \rip_*=\lip \calA(X), Y \rip$ an inner product, we can
choose the matrices $A_i$ and $B_i$ to be Hermitian. We start with a
calculation similar to the one we did above, which we record here for future use.

\begin{lemma}\label{le:adjointA}
Let $\calA : \Mn \to \Mn$ be given by 
$$
\calA (X) = \sum_{j=1}^m E_j X F_j,
$$ 
for matrices $E_j, F_j \in \Mn$. Then
$$
\lip X, \calA(Y) \rip = \tr\left( Y^* \sum_{j=1}^m  E_j^* X F_j^* \right).
$$
\end{lemma}
\begin{proof}
This is just a calculation:
\begin{align*}
\lip X, \calA(Y) \rip 
&= \tr\left( ( \calA(Y))^* X \right)\\
&= \tr\left( \sum_{j=1}^m  F_j^* Y^* E_j^* X \right)\\
&= \tr\left( \sum_{j=1}^m  Y^* E_j^* X F_j^* \right)\\
&= \tr\left( Y^* \sum_{j=1}^m  E_j^* X F_j^* \right).
\end{align*}
\end{proof}

The following is the main theorem of this section. It characterizes
the selfadjoint linear transformations from $\Mn$ to $\Mn$.

\begin{theorem}\label{th:hermitian}
Let $\calA : \Mn \to \Mn$ be the linear transformation defined, for
all $X \in \Mn$, as
$$
\calA(X)=\sum_{j=1}^m E_j X F_j,
$$
where the sets $\{ E_1, E_2, \dots, E_m \}$ and $\{ F_1, F_2, \dots,
F_m \}$ are linearly independent.

If for all $X, Y \in \Mn$ we have
$$
\lip \calA (X), Y \rip = \lip X, \calA(Y) \rip,
$$
then there exist $m \in \N$ and Hermitian matrices $A_j$ and $B_j$
in $\Mn$, for $j=1, 2, \dots, m$, such that
$$
\calA(X)=\sum_{j=1}^m A_j X B_j
$$
for all $X \in \Mn$. Furthermore, the sets $\{A_1, A_2, \dots, A_m\}$
and $\{ B_1, B_2, \dots, B_m\}$ are linearly independent.
\end{theorem}
\begin{proof}
First of all observe that by Lemma~\ref{le:adjointA}, for each $X, Y
\in \Mn$, we have
$$
\lip X, \calA(Y) \rip = \tr\left(
  Y^* \sum_{j=1}^m  E_j^* X F_j^* \right).
$$
Since, by hypothesis, 
$$
\lip \calA (X), Y \rip = \lip X, \calA(Y) \rip
$$
for all $X$ and $Y$, this implies that
$$
\tr\left(Y^* \sum_{j=1}^m  E_j X F_j \right)
= \tr\left(Y^* \sum_{j=1}^m  E_j^* X F_j^* \right)
$$
for all $Y \in \Mn$ and hence, by Proposition~\ref{prop:ortho}, we have
$$
\sum_{j=1}^m  E_j X F_j = \sum_{j=1}^m  E_j^* X F_j^*
$$
for all $X \in \Mn$.

By Lemma~\ref{le:eq}, there exist constants $c_{j,k} \in \C$ such that
$$
E_k^*=\sum_{j=1}^m \conj{c_{k,j}} E_j
$$
and 
$$
F_k^*=\sum_{i=1}^m c_{i,k} F_i,
$$
for $k=1, 2, \dots, m$.

Taking adjoints in the first expression above, and renaming indices, we obtain
$$
E_i=\sum_{k=1}^m c_{i,k} E_k^*.
$$
Combining both expressions we get
$$
E_i=\sum_{k=1}^m c_{i,k} E_k^*= \sum_{k=1}^m c_{i,k} \sum_{j=1}^m
\conj{c_{k,j}} E_j = \sum_{j=1}^m \sum_{k=1}^m c_{i,k} \conj{c_{k,j}} E_j.
$$
Since the set $\{E_1, E_2, \dots, E_m \}$ is linearly independent, it
follows that
$$
\sum_{k=1}^m c_{i,k} \conj{c_{k,j}} = \delta_{i,j}
$$
for each $i, j \in \{ 1, 2, \dots, m \}$. That is, if we define the
matrix $C \in M_m(\C)$ as $C=\begin{pmatrix} c_{i,j} \end{pmatrix}$,
we have that $C \conj{C}=I$.

By Theorem~\ref{th:coninv}, there exists a matrix $D$ such that $D
\conj{D}=I$ and $D^2=C$. From this, it follows that $D \conj{C} =
\conj{D}$. If $D=\begin{pmatrix} d_{i,j} \end{pmatrix}$, this means
that
$$
\sum_{k=1}^m d_{i,k} \conj{c_{k,j}} = \conj{d_{i,j}}
$$
and then
$$
\sum_{k=1}^m d_{i,k} E_k^* 
= \sum_{k=1}^m d_{i,k} \sum_{j=1}^m \conj{c_{k,j}} E_j 
= \sum_{j=1}^m \sum_{k=1}^m d_{i,k} \conj{c_{k,j}} E_j
=  \sum_{j=1}^m \conj{d_{i,j}} E_j.
$$
Thus, for each $k=1, 2, \dots, m$, the matrix
$$
A_k:=\sum_{j=1}^m \conj{d_{k,j}} E_j 
$$
is Hermitian.

Analogously, since $C \conj{D} = D$, it follows that
$$
\sum_{k=1}^m c_{i,k} \conj{d_{k,j}}= d_{i,j}
$$
and hence
$$
\sum_{k=1}^m \conj{d_{k,j}} F_k^* 
= \sum_{k=1}^m \conj{d_{k,j}} \sum_{i=1}^m c_{i,k} F_i
= \sum_{i=1}^m \sum_{k=1}^m c_{i,k} \conj{d_{k,j}}  F_i
= \sum_{i=1}^m d_{i,j} F_i.
$$
Thus, for each $k=1, 2, \dots, m$, the matrix
$$
B_k:=\sum_{i=1}^m d_{i,k} F_i 
$$
is Hermitian.

Since $D \conj{D}= I$, we have
$$
\sum_{k=1}^m d_{i,k} \conj{d_{k,j}} = \delta_{i,j}.
$$
Now, observe that, for each $X \in \Mn$ we have
\begin{align*}
\sum_{k=1}^m A_k X B_k 
&= \sum_{k=1}^m \left( \sum_{j=1}^m \conj{d_{k,j}} E_j \right) X \left(
  \sum_{i=1}^m d_{i,k} F_i \right) \\
&=  \sum_{i=1}^m \sum_{j=1}^m \sum_{k=1}^m d_{i,k} \conj{d_{k,j}}  E_j
X F_i \\
&=  \sum_{j=1}^m  E_j X F_j = \calA(X).
\end{align*}
The linear independence of the sets $\{ A_1, A_2, \dots, A_m \}$ and
$\{B_1, B_2, \dots, B_m \}$ follows easily from their definition and
the invertibility of $D$.
\end{proof}

As a consequence of the previous theorem, we have that for any given
inner product $\lip \cdot, \cdot \rip_*$ on $\Mn$ there exists $m \in
\N$ and Hermitian matrices $A_i$ and $B_i$, with $i=1, 2, \dots, m$
such that
$$
\lip X, Y \rip_* = \tr\left(\sum_{i=1}^m Y^* A_i X B_i \right)
$$
for all $X$ and $Y \in \Mn$. Can the matrices $A_i$ and $B_i$ be
chosen to be positive? We deal with that question in the next section.

\subsection*{Note:} The referee has observed that one obtains a weaker
conclusion  of the above theorem by observing that, in the above
proof, once we know the equality
$$
\sum_{j=1}^m  E_j X F_j = \sum_{j=1}^m  E_j^* X F_j^*
$$
for all $X \in \Mn$, it then follows that
$$
\calA(X)= \frac{1}{2} \sum_{j=1}^m  E_j X F_j + E_j^* X F_j^*.
$$
A computation then shows that 
$$
\calA(X)= \sum_{j=1}^m  \left(\frac{E_j+E_j^*}{2}\right) X \left(\frac{F_j +F_j^*}{2}\right) +
\left(\frac{E_j-E_j^*}{-2i}\right) X \left(\frac{F_j-F_j^*}{2i}\right).
$$
In the expression above, each of the matrices that multiply $X$ is
Hermitian and the conclusion of the theorem follows with $2m$ summands instead of $m$,
without using Theorems \ref{th:coninv}, \ref{th:fs1}, \ref{th:fs2} or
Lemma \ref{le:eq}. As the referee noted, our claim is stronger. We
appreciate the comment and thank the referee for the observation.


\section{Positive Matrices}\label{sec:pos}

Let $m>0$. It is clear that if we have a collection of positive matrices $A_i$,
$B_i$, for $i=1, 2, \dots, m$, then the operator $\calA$ defined by
$$
\calA (X) = \sum_{i=1}^m A_i X B_i 
$$
has the property that
$$
\lip \calA (X), X \rip >0
$$
for all nonzero $X \in \Mn$. Indeed,
\begin{align*}
\lip \calA (X), X \rip = \sum_{i=1}^m \tr(X^* A_i X B_i) 
\end{align*}
and each term is positive by Lemma~\ref{le:tra_pos}.

The natural question arises: is the converse true? That is, if $\calA
: \Mn \to \Mn$ is defined as
$$
\calA (X) = \sum_{i=1}^m A_i X B_i 
$$
and $\lip \calA(X), X \rip >0 $ for all nonzero $X\in \Mn$, can we
choose the matrices $A_i$ and $B_i$ to be positive? If $\calA$
consists of one summand, the answer is affirmative.

\begin{theorem}\label{th:main1}
Let $\calA: \Mn \to \Mn$ be given by $\calA (X) = A X B$ with $A, B
\in \Mn$ Hermitian. If $\lip \calA(X), X \rip >0$ for all $X \in
\Mn$, $X \neq 0$, then there exist positive matrices $A'$ and $B'$ in
$\Mn$ such that
$$
\calA (X) = A' X B'.
$$
\end{theorem}
\begin{proof}
Let $x$ and $y$ be nonzero vectors in $\C^n$. Then
$$
\lip \calA(x y^*), x y^* \rip > 0
$$
and hence
$$
\lip \calA(x y^*), x y^* \rip = (x^* A x) (y^* B y) > 0.
$$

If $y^* B y > 0$ for some nonzero $y \in \C^n$, this means that $x^* A x >0$
for all nonzero $x \in \C^n$ and hence that $A$ is positive. In turn, this
implies that $y^* B y > 0$ for all nonzero $y \in \C^n$ and hence $B$ is
positive. Thus the theorem is proved with $A'=A$ and $B'=B$.

On the other hand, if $y^* B y < 0$ for some nonzero $y \in \C^n$,
this means that $x^* A x < 0$ for all nonzero $x \in \C^n$ and hence that $-A$
is positive. In turn, this implies that $y^* B y < 0$ for all nonzero $y \in
\C^n$ and hence $-B$ is positive. Thus the theorem is proved with
$A'=-A$ and $B'=-B$.
\end{proof}

\subsection{Two or more summands}

The rest of this paper will be devoted to try to answer the question
above in the case of two or more summands. The
Proposition~\ref{prop:Bpos} below gives a partial answer.

In the next proposition, if $m=2$ whenever we write $\ds \sum_{i=3}^m
\square $ we assume the sum is zero.

\begin{proposition}\label{prop:Bpos}
Let $m \in \N$ with $m\geq 2$. Let $\calA : \Mn \to \Mn$ be given by
$$
\calA(X) = \sum_{i=1}^m A_i X B_i,
$$
with $A_i, B_i$ Hermitian.  If $\lip \calA(X), X \rip >0$ for all
nonzero $X \in \Mn$ then there exist positive matrices $B_i'$ for
$i=1, 2, \dots m$, a positive matrix $A_1'$ and Hermitian matrices
$A_i'$ for $i=2, 3, \dots, m$, such that
$$
\calA(X) = \sum_{i=1}^m A_i' X B_i', \quad \hbox{ for all } X \in M_n.
$$
\end{proposition}
\begin{proof}
We first show that one of the matrices $B_i$ can be chosen to be
positive. Fix a nonzero vector $x_0 \in \C^n$. For all nonzero $y \in
\C^n$ we have
\begin{align*}
0 < \lip \calA(x_0 y^*), x_0 y^* \rip 
&= \sum_{i=1}^m (x_0^* A_i x_0) (y^* B_i y)\\
&=  \sum_{i=1}^m \alpha_i (y^* B_i y)\\
&=  y^* \left( \sum_{i=1}^m \alpha_i B_i \right) y,
\end{align*}
with $\alpha_i:=x_0^* A_i x_0 \in \R$ for $i=1,2, \dots, m$. Hence,
$\ds \sum_{i=1}^m \alpha_i B_i$ is positive. This in turn implies that
not all $\alpha_i$ equal zero. Assume, without loss of generality,
that $\alpha_1\neq 0$. Then, we can write $\calA$ as
$$
\calA(X) = \left(\frac{A_1}{\alpha_1}\right) X \left(\sum_{i=1}^m \alpha_i B_i \right) 
+ \sum_{i=2}^m \left(A_i -\frac{\alpha_i}{\alpha_1} A_1\right) X B_i.
$$

By renaming, we can then assume that
$$
\calA (X) = \sum_{i=1}^m A_i X B_i 
$$
with $B_1$ positive and the rest of the matrices Hermitian.

Since $B_1$ is positive, we can choose a constant $\beta>0$
sufficiently large such that $B_2 + \beta B_1$ is a positive
matrix. We can then write $\calA$ as
$$
\calA(X)= (A_1 -\beta A_2)X B_1 + A_2 X (B_2 + \beta B_1) +
\sum_{i=3}^m A_i X B_i,
$$
with $B_1$ and $B_2+\beta B_1$ positive. By renaming, we may now
assume that 
$$
\calA (X) = \sum_{i=1}^m A_i X B_i 
$$
with $B_1$ and $B_2$ both positive and the rest of the matrices Hermitian.

Consider the family of Hermitian matrices $\{ B_1 - tB_2 \}$ for $t\geq
0$. By continuity, there exists $t_0>0$ such that $B_1 - t B_2$ is
positive for all $t\in[0,t_0)$ and $B_1 -t_0B_2$ has a zero eigenvalue,
with eigenvector $y_0\in \C^n$.

It then follows that, for nonzero $x \in \C^n$
\begin{align*}
0 < \lip \calA(x y_0^*), x y_0^* \rip 
&=\sum_{i=1}^m (x^* A_i x) (y_0^* B_i y_0) \\
&= (x^* A_1 x) (y_0^* (B_1 - t_0 B_2) y_0) + (x^* (A_2 + t_0 A_1) x)
(y_0^* B_2 y_0) \\
& \quad \qquad \qquad + \sum_{i=3}^m (x^* A_i x) (y_0^* B_i y_0)\\
&=  (x^* (A_2 + t_0 A_1) x) (y_0^* B_2 y_0) + \sum_{i=3}^m (x^* A_i x) (y_0^* B_i y_0)\\
&= x^* \left( \gamma_2 (A_2 + t_0 A_1) +  \sum_{i=3}^m \gamma_i A_i \right) x
\end{align*}
with $\gamma_i:= y_0^* B_i y_0 \in\R$ for $i=2, 3, \dots, m$. Observe that
$\gamma_2 >0$. 

We have shown that
$$
\gamma_2 (A_2 + t_0 A_1) +  \sum_{i=3}^m \gamma_i A_i
$$
is positive. By continuity, for $\varepsilon>0$ small enough we also have
$$
\gamma_2 (A_2 + (t_0-\varepsilon) A_1) +  \sum_{i=3}^m \gamma_i A_i
$$
is positive. In fact, we can choose $\varepsilon >0$ small enough to
guarantee that
$$
B_1- (t_0 - \varepsilon)B_2
$$
is also positive.

Now we can use this to write $\calA$ as
\begin{align*}
\calA(X) 
&= A_1 X (B_1-(t_0-\varepsilon)B_2)\\
&\quad \qquad \qquad + \left(\gamma_2 (A_2 + (t_0-\varepsilon) A_1) +
  \sum_{i=3}^m \gamma_i A_i \right) X \frac{B_2}{\gamma_2} \\
&\qquad \qquad \qquad + \sum_{i=3}^m A_i X \left(B_i - \frac{\gamma_i}{\gamma_2}B_2 \right)
\end{align*}

Again, by renaming, we may assume that $\calA$ is of the form
$$
\calA(X)=\sum_{i=1}^m A_i X B_i
$$
with $A_2$, $B_1$ and $B_2$ positive and the rest of the matrices Hermitian.

Now, since $B_1$ is positive, we can choose constants $\beta_i >0$ sufficiently large such that
$$
B_i + \beta_i B_1
$$
is positive for $i=3, \dots, m$. We can then write $\calA$ as
\begin{align*}
\calA(X) 
&= \left(A_1- \sum_{i=3}^m \beta_i A_i \right)X B_1 + 
 A_2 XB_2  \\
& \quad \qquad \qquad + A_3 X (B_3+\beta_3 B_1)+ \dots  + A_m X(B_m+\beta_m B_1),
\end{align*}
and thus $\calA$ has the form
$$
\calA(X) = \sum_{i=1}^m A_i X B_i,
$$
with $A_i$ Hermitian for all $i$, $A_2$ positive and $B_i$ positive
for all $i$. Renaming the matrices we obtain the desired conclusion.
\end{proof}

For the case where there are exactly two summands, we obtain the
result that all inner products come from positive matrices.

\begin{theorem}\label{th:main2}
Let $\calA: \Mn \to \Mn$ be given by 
$$
\calA (X) = A_1 X B_1 + A_2 X B_2
$$ 
with $A_i, B_i \in \Mn$ Hermitian. If $\lip \calA(X), X \rip >0$ for
all nonzero $X \in \Mn$, then there exist positive matrices $A_1',
A_2', B_1', B_2'$ such that
$$
\calA (X) = A_1' X B_1' + A_2' X B_2'.
$$
\end{theorem}
\begin{proof}
By Proposition~\ref{prop:Bpos} we may assume that $\calA$ is of the form
$$
\calA(X) = A_1 X B_1  + A_2 X B_2 
$$
with $B_1$, $B_2$ and $A_1$ positive and $A_2$ Hermitian.

Consider the family of Hermitian matrices $B_1 - t B_2$ for $t \geq
0$. Again, there exists a point $t_0>0$ such that $B_1 - t B_2$ is
positive for all $t \in [0, t_0)$ and $B_1 - t_0 B_2$ has zero as an
eigenvalue with eigenvector $y_0$. As before, 
\begin{align*}
0 < \lip \calA(x y_0^*), x y_0^* \rip 
&= (x^* A_1 x) (y_0^* B_1 y_0) + (x^* A_2 x) (y_0^* B_2 y_0) \\
&=  (x^* A_1  x) (y_0^* (B_1-t_0 B_2) y_0) + (x^*(A_2+t_0 A_1) x) (y_0^* B_2 y_0) \\
&= (x^*(A_2+t_0 A_1) x) (y_0^* B_2y_0) \\
\end{align*}
and hence $A_2+t_0 A_1$ is positive.  By continuity, there
exists $\varepsilon >0$ small enough such that both
$$
B_1 - (t_0 - \varepsilon) B_2 \qquad \text{ and } \qquad  A_2 + (t_0 - \varepsilon) A_1 
$$
are positive.

Thus we can write $\calA$ as
$$
\calA(X) = A_1 X (B_1-(t_0-\varepsilon)B_2) + (A_2+(t_0-\varepsilon)A_1) X B_2 ,
$$
with 
$$
A_1, \quad B_1-(t_0-\varepsilon)B_2, \quad A_2+(t_0-\varepsilon)A_1, \quad \text{ and } \quad \quad B_2
$$ 
all positive.

Thus we conclude that we can write $\calA$ as
$$
\calA(X) = A_1 X B_1  + A_2 X B_2 
$$
with $A_1$, $B_1$, $A_2$ and $B_2$ all positive.
\end{proof}

We would like to extend the theorem above for an arbitrary number of
summands in $\calA$. Nevertheless, this is not possible, as the
following example shows. 

\subsection{Counterexample}\label{subsec:2.1}

The following example is a restatement of the example shown to us by
D.~Grinberg in~\cite{grinberg}.

\begin{example}\label{example}
Let $\varepsilon \in (0,\frac{1}{2})$ and let $\calA : M_2 (\C) \to M_2(\C)$ be given by
$$
\calA \begin{pmatrix} x & y \\ w & z \end{pmatrix}=\begin{pmatrix} x +
  (1-\varepsilon) z & \varepsilon y \\ \varepsilon w & z +
  (1-\varepsilon) x \end{pmatrix}.
$$
Then,
\begin{itemize}
\item $\lip \calA(X), X \rip >0$ for all nonzero $X \in M_2(\C)$.
\item There do not exist $m \in \N$ and positive matrices $A_i, B_i
  \in M_2(\C)$ for $i=1, 2,
\dots, m$ such that
$$
\calA(X) = \sum_{i=1}^m A_i X B_i
$$
for all $X \in M_2(\C)$.
\end{itemize}
\end{example}
\begin{proof}
To prove the first part, just observe that if 
$$ 
X=\begin{pmatrix} x & y \\ w & z \end{pmatrix},
$$
then 
$$
\lip  \calA(X), X\rip = \varepsilon (|x|^2+|y|^2+|w|^2+|z|^2) +
(1-\varepsilon) |x +z|^2
$$
and hence $\lip \calA(X), X \rip >0$ for all nonzero matrices $X \in M_2(\C)$.

For the second part, assume that there exist $m \in \N$ and positive matrices 
$$
A_i=\begin{pmatrix}
\alpha_i & \gamma_i \\
\conj{\gamma_i} & \beta_i
\end{pmatrix} \quad \text{ and } \quad B_i=\begin{pmatrix}
\alpha_i' & \gamma_i' \\
\conj{\gamma_i'} & \beta_i'
\end{pmatrix}
$$
for $i=1, 2, \dots, m$ such that
$$
\calA(X)=\sum_{i=1}^m A_i X B_i.
$$
Since each $A_i$ and $B_i$ are positive, it follows that $\alpha_i$,
$\alpha_i'$, $\beta_i$ and $\beta_i'$ are all positive numbers. Also,
$\alpha_i \beta_i > |\gamma_i|^2$ and $\alpha_i' \beta_i' >
|\gamma_i'|^2$.

Now, we have
$$
\calA \begin{pmatrix} 0 & 1 \\ 0 & 0 \end{pmatrix} = \begin{pmatrix} 0
  & \varepsilon \\ 0 & 0 \end{pmatrix}, \qquad \text{ and } \quad
\calA \begin{pmatrix} 0 & 0 \\ 1 & 0 \end{pmatrix} = \begin{pmatrix} 0
  & 0 \\ \varepsilon & 0 \end{pmatrix}.
$$

But the assumption implies that
$$
\calA \begin{pmatrix} 0 & 1 \\ 0 & 0 \end{pmatrix} = \begin{pmatrix}
\sum_{i=1}^m \alpha_i \conj{\gamma_i'} & \sum_{i=1}^m \alpha_i \beta_i' \\
\sum_{i=1}^m \conj{\gamma_i}\conj{\gamma_i'} & \sum_{i=1}^m \conj{\gamma_i} \beta_i'
\end{pmatrix}
$$
and
$$
\calA \begin{pmatrix} 0 & 0 \\ 1 & 0 \end{pmatrix} = \begin{pmatrix} 
\sum_{i=1}^m \gamma_i \alpha_i' & \sum_{i=1}^m \gamma_i \gamma_i' \\
\sum_{i=1}^m \beta_i \alpha_i' & \sum_{i=1}^m \beta_i \gamma_i'
\end{pmatrix}.
$$
Hence, 
$$
\varepsilon= \sum_{i=1}^m \alpha_i \beta_i'= \sum_{i=1}^m \beta_i \alpha_i'.
$$
Thus
$$
\varepsilon = \sum_{i=1}^m \frac{\alpha_i \beta_i' + \beta_i
  \alpha_i'}{2} \geq \sum_{i=1}^m \sqrt{\alpha_i \beta_i' \beta_i
  \alpha_i'} \geq \sum_{i=1}^m | \gamma_i | \, |\gamma_i'| \geq \left|
  \sum_{i=1}^m  \gamma_i \conj{\gamma_i'} \right|
$$
by the arithmetic mean--geometric mean inequality, due to the positiviy of
$\alpha_i$, $\alpha_i'$, $\beta_i$ and $\beta_i'$ and the inequalities
$\alpha_i \beta_i > |\gamma_i|^2$ and $\alpha_i' \beta_i' > |\gamma_i'|^2$.

But observe that on one hand
$$
\calA \begin{pmatrix} 0 & 0 \\ 0 & 1 \end{pmatrix} = \begin{pmatrix} 1-\varepsilon
  & 0 \\ 0 & 1 \end{pmatrix}
$$
and on the other
$$
\calA \begin{pmatrix} 0 & 0 \\ 0 & 1 \end{pmatrix} = \begin{pmatrix} 
\sum_{i=1}^m \gamma_i \conj{\gamma_i'} & \sum_{i=1}^m \gamma_i \beta_i' \\
\sum_{i=1}^m \beta_i \conj{\gamma_i'} & \sum_{i=1}^m \beta_i \beta_i'
\end{pmatrix},
$$
and hence $\sum_{i=1}^m \gamma_i \conj{\gamma_i'}= 1 - \varepsilon$. 

The previous inequality implies that 
$$
\varepsilon \geq \left| \sum_{i=1}^m  \gamma_i \conj{\gamma_i'} \right| = 1 - \varepsilon
$$
which implies that $\varepsilon \geq \frac{1}{2}$, contradicting the hypothesis.
\end{proof}

Observe that the map $\calA$ in the example above can be written as
\begin{align*}
\calA(X) 
&= \begin{pmatrix} 1 & 0 \\ 0 & \varepsilon \end{pmatrix} X
\begin{pmatrix} 1 & 0 \\ 0 & 0 \end{pmatrix} +
\begin{pmatrix} 0 & 1-\varepsilon \\ 0 & 0 \end{pmatrix} X
\begin{pmatrix} 0 & 0 \\ 1 & 0 \end{pmatrix}\\
&+
\begin{pmatrix} 0 & 0 \\ 1-\varepsilon & 0 \end{pmatrix} X
\begin{pmatrix} 0 & 1 \\ 0 & 0 \end{pmatrix} +
\begin{pmatrix} \varepsilon & 0 \\ 0 & 1 \end{pmatrix} X
\begin{pmatrix} 0 & 0 \\ 0 & 1 \end{pmatrix},
\end{align*}
and it is not hard to show that it cannot be written as the sum of
less summands.

\subsection{Three or more summands}

As the previous example shows, Theorems~\ref{th:main1}
and~\ref{th:main2} cannot be extended to $four$ or more
summands. Nevertheless, something can be said.

\begin{theorem}\label{th:3sum}
Let $m \in \N$ and $m \geq 2$. Let $\calA: \Mn \to \Mn$ be given by 
$$
\calA (X) = \sum_{i=1}^m A_i X B_i 
$$ 
with $A_i, B_i \in \Mn$ Hermitian. If $\lip \calA(X), X \rip >0$ for
all $X \in \Mn$, $X \neq 0$, then there exist positive matrices $A_i',
B_i'$ such that
$$
\calA (X) = - A_1' X B_1' + \sum_{i=2}^m A_i' X B_i',
$$
for all $X \in \Mn$.
\end{theorem}
\begin{proof}
By Proposition~\ref{prop:Bpos} we may assume that $\calA$ has the form 
$$
\calA (X) = \sum_{i=1}^m A_i X B_i 
$$
with $A_i$ Hermitian, $A_1$ positive, and $B_i$ positive, for each $i=1, 2, \dots,
m$.

Since $A_1$ is positive, we can choose $\alpha >0$ large enough such
that $\alpha A_1 - A_2$ is positive. Hence, we can write $\calA$ as
$$
\calA(X)= A_1 X (B_1 +\alpha B_2) - (\alpha A_1 - A_2)X B_2 + \sum_{i=3}^m A_i X B_i
$$
and hence we may assume that $\calA$ can be written as
$$
\calA(X) = A_1 X B_1 - A_2 X B_2 + \sum_{i=3}^m A_i X B_i
$$
with $B_i$ positive for each $i=1, 2, \dots, m$, with $A_1$ and $A_2$
positive, and $A_i$ Hermitian for $i=3, 4, \dots, m$. 

For $k=3, 4, \dots, m$, we can choose $\eta_k >0$ large enough such
that $A_k + \eta_k A_2$ is positive. Hence $\calA$ can be written as
$$
\calA(X) = A_1 X B_1 - A_2 X \left(B_2 +  \sum_{k=3}^m \eta_k B_k \right)
+ \sum_{k=3}^m (A_k + \eta_k A_2) X B_k
$$
and thus, by renaming, we have
$$
\calA(X) = -A_1' X B_1' + \sum_{i=2}^m A_i' X B_i'
$$
with $A_i'$ and $B_i'$ positive for $i=1, 2, \dots, m$.
\end{proof}

The previous theorem can be improved under certain conditions.

\begin{theorem}\label{th:main_pos}
Let $m \in \N$ and $m \geq 2$. Let $\calA: \Mn \to \Mn$ be given by 
$$
\calA (X) = - A_1 X B_1 + \sum_{i=2}^m A_i X B_i 
$$ 
with $A_i, B_i \in \Mn$ positive. If there exist positive numbers
$\xi_k$ for $k=2, 3, \dots, m$ such that 
$$
-A_1 + \sum_{k=2}^m \xi_k A_k
$$
is positive and $B_k-\xi_kB_1$ is positive for all $k=2, 3, \dots, m$ 
then there exist positive matrices $A_i', B_i'$ such that
$$
\calA (X) = \sum_{i=1}^m A_i' X B_i',
$$
for all $X \in \Mn$.
\end{theorem}
\begin{proof}
If the hypothesis hold, we can write $\calA$ as
$$
\calA(X) = \left( -A_1 + \sum_{k=2}^m \xi_k A_k \right) X B_1 +
\sum_{k=2}^m A_k X (B_k - \xi_k B_1),
$$
which, by renaming, is the desired conclusion.
\end{proof}

Perhaps it should be noted that one can use the theorems above to
obtain, in explicit form, appropriate expressions for $\calA$, at
least in simple cases.

For example, if we define $\calA : M_2(\C) \to M_2(\C)$ as in
Example~\ref{example}, one can use the procedure in
Theorem~\ref{th:hermitian} to show that $\calA$ can be written as
\begin{align*}
\calA(X)
&= \begin{pmatrix} 1 & 0 \\ 0 & \varepsilon \\ \end{pmatrix} X
\begin{pmatrix} 1 & 0 \\ 0 & 0 \\ \end{pmatrix}
+ (1-\varepsilon)
\begin{pmatrix} 0 & \frac{1-i}{2} \\ \frac{1+i}{2} & 0 \\ \end{pmatrix} X
\begin{pmatrix} 0 & \frac{1-i}{2} \\ \frac{1+i}{2} & 0 \\ \end{pmatrix}\\
& \qquad \hbox{ } \\
& \qquad + (1-\varepsilon)
\begin{pmatrix} 0 & \frac{1+i}{2} \\ \frac{1-i}{2} & 0 \\ \end{pmatrix} X
\begin{pmatrix} 0 & \frac{1+i}{2} \\ \frac{1-i}{2} & 0 \\ \end{pmatrix}
+ \begin{pmatrix} \varepsilon & 0 \\ 0 & 1 \\ \end{pmatrix} X 
\begin{pmatrix} 0 & 0 \\ 0 & 1 \\ \end{pmatrix}.
\end{align*}

Also, we can apply the procedure of Proposition~\ref{prop:Bpos} and of
Theorem~\ref{th:3sum}, for the case $\varepsilon=\frac{1}{4}$, to write $\calA$ as
\begin{align*}
\calA(X)
&= - \frac{1}{32}\begin{pmatrix}1&2\\ 2& 16 \end{pmatrix}  X
\begin{pmatrix}279&48\\48&36\end{pmatrix} 
+ \frac{1}{32}\begin{pmatrix}3&9-i\\ 9+i&56\end{pmatrix} X
\begin{pmatrix}31&6-6i\\6+6i &4\end{pmatrix} \\
& \quad + \frac{1}{32}\begin{pmatrix}3&9+i\\ 9-i&56\end{pmatrix} X
\begin{pmatrix}31&6+6i\\6-6i&4\end{pmatrix}
+ \frac{1}{32}\begin{pmatrix}1&0\\0&8\end{pmatrix} X
\begin{pmatrix}125&12\\12&20\end{pmatrix}.
\end{align*}

\section{Further questions}

We believe the following questions are worth further inquiry:
\begin{itemize}
\item Can Theorem~\ref{th:main2} be extended to $3$ summands?
\item We know that Theorem~\ref{th:main2} cannot be extended to $4$
  summands. For what $m$ can the theorem be extended?  
\item For what dimensions is Theorem \ref{th:main2} true for all
  number of summands?
\item Theorem~\ref{th:main_pos} gives conditions under which
  Theorem~\ref{th:main2} can be extended. Nevertheless, the hypotheses
  are not simple to check. Is there an ``easy-to-compute'' condition on
  the matrices $A_i$ and $B_i$ that guarantees that
  Theorem~\ref{th:main2} can be extended to any number of summands?
\item Is there an ``easy-to-compute'' necessary and sufficient condition on $\calA$ that
  guarantees that it can be written as a sum of the form
$$
\calA(X)=\sum_{i=1}^m A_i X B_i
$$
with $A_i$ and $B_i$ positive?
\end{itemize}


\end{document}